\documentclass[10pt,reqno,a4paper,oneside,11pt]{amsart}%
\usepackage{amsfonts}
\usepackage{amsfonts}
\usepackage{amsfonts}
\usepackage{amsfonts}
\usepackage{amsfonts}
\usepackage{amsfonts}
\usepackage{amsfonts}
\usepackage{amsfonts}
\usepackage{amsfonts}
\usepackage{amsfonts}
\usepackage{amsfonts}
\usepackage{amsfonts}
\usepackage{amsfonts}
\usepackage{amsfonts}
\usepackage{amsfonts}
\usepackage{amsfonts}
\usepackage{amsfonts}
\usepackage{amsfonts}
\usepackage{amsfonts}
\usepackage{amsfonts}
\usepackage{mathrsfs}
\usepackage{mathrsfs}
\usepackage{amsfonts}
\usepackage{amssymb}
\usepackage{amsmath}
\usepackage{amsthm}
\usepackage{graphicx}
\providecommand{\U}[1]{\protect\rule{.1in}{.1in}}
\newtheorem{theorem}{Theorem}[section]
\newtheorem{corollary}[theorem]{Corollary}
\newtheorem{lemma}[theorem]{Lemma}

\newtheorem{example}{Example}

\theoremstyle{definition}
\theoremstyle{remark}
\numberwithin{equation}{section}

\ifx\pdfoutput\relax\let\pdfoutput=\undefined\fi
\newcount\msipdfoutput
\ifx\pdfoutput\undefined\else
\ifcase\pdfoutput\else
\ifx\paperwidth\undefined\else
\ifdim\paperheight=0pt\relax\else\pdfpageheight\paperheight\fi
\ifdim\paperwidth=0pt\relax\else\pdfpagewidth\paperwidth\fi
\fi\fi\fi
\begin{document}
\pagestyle{myheadings}

\begin{center}
{\huge \textbf{A simultaneous decomposition of four real quaternion matrices encompassing $\eta$-Hermicity and its applications}}
\footnote{This research was supported by
the grants from the National Natural
Science Foundation of China (11571220).
\par
{}* Corresponding author. hzh19871126@126.com (Z.H. He),  wqw@t.shu.edu.cn, wqw369@yahoo.com (Q.W. Wang)}

\bigskip

{ \textbf{Zhuo-Heng He$^{a,b}$, Qing-Wen Wang$^{b,*}$}}

{\small
\vspace{0.25cm}

$a.$ Department of Mathematics and Statistics, Auburn University, AL 36849-5310, USA\\

$b.$ Department of Mathematics, Shanghai University, Shanghai 200444, P. R. China}

\end{center}

\begin{quotation}
\noindent\textbf{Abstract:} Let $\mathbb{H}$ be the real quaternion algebra and $\mathbb{H}^{m\times n}$ denote the set of all $m\times n$ matrices over $\mathbb{H}$. Let $\mathbf{i},\mathbf{j},\mathbf{k}$ be the imaginary quaternion units.  For $\eta\in\{\mathbf{i},\mathbf{j},\mathbf{k}\}$, a square real quaternion matrix $A$ is said to be $\eta$-Hermitian if $A^{\eta*}=A$ where $A^{\eta*}=-\eta A^{\ast}\eta$, and
$A^{\ast}$ stands for the conjugate transpose of $A$.  In this paper,  we  construct a  simultaneous decomposition of four real quaternion matrices with the same row number
$(A,B,C,D),$ where $A=A^{\eta*}\in \mathbb{H}^{m\times m},
B\in \mathbb{H}^{m\times p_{1}},C\in \mathbb{H}^{m\times p_{2}},D\in \mathbb{H}^{m\times p_{3}}$. As applications of  this simultaneous
matrix decomposition,   we derive necessary and sufficient conditions for some real quaternion matrix equations involving $\eta$-Hermicity in terms of
ranks of the coefficient matrices. We also present the general solutions to these  real quaternion matrix equations. Moreover, we provide some
numerical examples to illustrate our results.
\newline\noindent\textbf{Keywords:}   Matrix decomposition; Matrix equation; Quaternion; Solvability; General $\eta$-Hermitian solution; Rank\newline%
\noindent\textbf{2010 AMS Subject Classifications:\ }{\small 15A09, 15A23, 15A24, 15B33, 15B57}\newline
\end{quotation}

\section{\textbf{Introduction}}

Let $\mathbb{R}$ and $\mathbb{H}^{m\times
n}$ stand, respectively, for the real number field and the set of all $m\times n$ matrices over the real quaternion algebra
\[
\mathbb{H}%
=\big\{a_{0}+a_{1}\mathbf{i}+a_{2}\mathbf{j}+a_{3}\mathbf{k}\big|~\mathbf{i}^{2}=\mathbf{j}^{2}=\mathbf{k}^{2}=\mathbf{ijk}=-1,a_{0}%
,a_{1},a_{2},a_{3}\in\mathbb{R}\big\}.
\]
For $\eta\in\{\mathbf{i},\mathbf{j},\mathbf{k}\}$, a square real quaternion matrix $A$ is said to be $\eta$-Hermitian if $A^{\eta*}=A$ where $A^{\eta*}=-\eta A^{\ast}\eta$, and
$A^{\ast}$ stands for the conjugate transpose of $A$ (\cite{C. Cheong Took4}). The map $A\mapsto A^{\eta\ast}$  on $\mathbb{H}^{n\times n}$ is involutorial (\cite{wanghe2}). The symbol $r(A)$ stands for the rank of a given real quaternion matrix $A$. For a real quaternion matrix $A$, $r(A)=r(A^{\eta*})$ \cite{wanghe2}.   The identity matrix and zero matrix with appropriate sizes
are denoted by $I$ and $0$, respectively. The set of all $n\times n$ invertible matrix over $\mathbb{H}$ are denoted by $GL_{n}.$

The $\eta$-Hermitian matrices arise in
statistical signal processing and widely linear modelling  (\cite{C. Cheong Took1}-\cite{C. Cheong Took4}).  The decompositions of matrices and $\eta$-Hermitian matrices have applications in system and control theory, signal processing, linear modelling, engineering and so on (e.g., \cite{CHU3}-\cite{moorbelgium}, \cite{golub01}, \cite{alter03}-\cite{stewart}, \cite{C. Cheong Took4}). The study on the decompositions of $\eta$-Hermitian matrices
is active in recent years. The decomposition of an $\eta$-Hermitian  matrix was first proposed in 2011 (\cite{C. Cheong Took4}).  Horn and Zhang \cite{FZhang3} presented an analogous special singular value decomposition for $\eta$-Hermitian matrices. Very recently, He and Wang \cite{hewangamc2017} gave a  simultaneous decomposition for a set of nine real quaternion  matrices involving $\eta$-Hermicity
with  compatible sizes: $A_{i}\in\mathbb{H}^{p_{i}\times t_{i}},B_{i}\in\mathbb{H}^{p_{i}\times t_{i+1}},$ and $C_{i}\in\mathbb{H}^{p_{i}\times p_{i}},$ where $C_{i}$ are $\eta$-Hermitian matrices, $(i=1,2,3)$.

To the best of our knowledge, there is little information on the simultaneous decomposition of four real quaternion matrices with the same row number involving $\eta$-Hermicity:
\begin{align}\label{array01}
\bordermatrix{
~& m & p_{1}&p_{2}&p_{3} \cr
m&A&B&C&D},
\end{align}
where $B\in \mathbb{H}^{m\times p_{1}},C\in \mathbb{H}^{m\times p_{2}},D\in \mathbb{H}^{m\times p_{3}},$ and $A\in \mathbb{H}^{m\times m}$ is an $\eta$-Hermitian matrix. Motivated by the wide application of real quaternion matrices and $\eta$-Hermitian matrices and in order to improve the theoretical development of the decompositions of $\eta$-Hermitian matrices, we consider the simultaneous decomposition of four real quaternion matrices involving $\eta$-Hermicity (\ref{array01}). One contribution of this paper is to show how to find matrices $P\in GL_{m}(\mathbb{H}), T_{1}\in GL_{p_{1}}(\mathbb{H}),T_{2}\in GL_{p_{2}}(\mathbb{H}),T_{3}\in GL_{p_{3}}(\mathbb{H}),$ such that
\begin{align}\label{dec12}
PA P^{\eta*}=S_{A},\qquad PB T_{1}=S_{B},\qquad PC T_{2}=S_{C},\qquad PDT_{3}=S_{D},
\end{align}
where $S_{B},S_{C},S_{D}$ are quasi-diagonal matrices with the finest possible subdivision of matrices, and $S_{A}=S_{A}^{\eta*}$ have appropriate forms (see Theorem \ref{theorem01} for the definitions in details). We
conjecture that this simultaneous decomposition  will also play an important role in signal processing and linear modelling.

Using the simultaneous matrix decomposition (\ref{dec12}), we consider the following two real quaternion matrix equations involving $\eta$-Hermicity:
\begin{align}\label{4system001}
BXB^{\eta*}+CYC^{\eta*}+DZD^{\eta*}=A, ~X=X^{\eta*},~Y=Y^{\eta*},~Z=Z^{\eta*}
\end{align}
and
\begin{align}\label{4system002}
BXC+(BXC)^{\eta*}+DYD^{\eta*}=A,~Y=Y^{\eta*}.
\end{align}
where $A,B,C,$ and $D$ are given real quaternion matrices, $X,Y,Z$ are unknowns. We will make use of the simultaneous matrix decomposition (\ref{dec12}) that bring the real quaternion matrix equations (\ref{4system001}) and (\ref{4system002}) to some canonical forms. Then we can give some necessary and sufficient conditions for the existence of the general solutions to the real quaternion matrix equations (\ref{4system001}) and (\ref{4system002}) in terms of the ranks of the given coefficient matrices. There have been many papers using different approaches to investigate the matrix equations and real quaternion matrix equations involving $\eta$-Hermicity
(e.g., \cite{xibanyalama}-\cite{farid}, \cite{helaa}-\cite{heac2017}, \cite{wanghe4444444}-\cite{wangronghao}).

The rest of this paper is organized as follows. We in Section 2 construct a  simultaneous decomposition of four real quaternion matrices involving $\eta$-Hermicity (\ref{array01}). As applications of this simultaneous decomposition, we in Section 3 establish necessary and sufficient
conditions for the existence of the $\eta$-Hermitian solution to the real quaternion matrix equation (\ref{4system001}), and give an expression of this $\eta$-Hermitian solution when the solvability conditions are satisfied. In Section 4, we derive necessary and sufficient
conditions for the existence of the solution to the real quaternion matrix equation (\ref{4system002}), and present an expression of the general solution when the solvability conditions are satisfied.

\section{\textbf{A simultaneous decomposition of four real quaternion matrices (\ref{array01})}}

In this section, we establish a  simultaneous decomposition of four real quaternion matrices involving $\eta$-Hermicity (\ref{array01}).  We begin with the following lemma that is an important
tool for obtaining the main result.
\begin{lemma}\label{lemma00}\cite{QWWangandyushaowen}
Let $B\in \mathbb{H}^{m\times p_{1}},
C\in \mathbb{H}^{m\times p_{2}}$ and $D\in \mathbb{H}^{m\times p_{3}}$ be given.  Then there exist
 $P_{1}\in GL_{m}(\mathbb{H}),$ $W_{B}\in GL_{p_{1}}(\mathbb{H}),$ $
W_{C}\in GL_{p_{2}}(\mathbb{H}),$ and $W_{D}\in GL_{p_{3}}(\mathbb{H})$ such that
\begin{align*}
P_{1}BW_{B}=\widetilde{S_{B}}, \qquad P_{1}CW_{C}=\widetilde{S_{C}},\qquad P_{1}DW_{D}=\widetilde{S_{D}},
\end{align*}
where
\begin{align}\label{requ021}
\widetilde{S_{B}}=
\begin{pmatrix}
I&0\\0&0
\end{pmatrix}
\begin{matrix}
r(B)\\~
\end{matrix},\quad
\widetilde{S_{C}}=
\begin{pmatrix}
0&I&0\\
0&0&0\\
I&0&0\\
0&0&0
\end{pmatrix}
\begin{matrix}
r_{2}\\r(B)-r_{2}\\r_{1}\\~
\end{matrix},
\end{align}
\begin{align}
\widetilde{S_{D}}=\begin{pmatrix}
0 &0&0&0 &I&0\\
0 & 0&0 &0&0&0\\
0&0&0&I&0&0\\
0&I & 0&0&0&0\\
0 & 0&0 &0&0&0\\
0&I & 0&0&0&0\\
0 & 0&I&0&0&0\\
0 & 0&0 &0&0&0\\
I&0 & 0&0&0&0\\
0 & 0&0 &0&0&0
\end{pmatrix}
\begin{matrix}
r_{6}\\
r_{2}-r_{6}\\
r_{5}\\
r_{7}\\
r(B)-r_{2}-r_{5}-r_{7}\\
r_{7}\\
r_{4}-r_{7}\\
r_{1}-r_{4}\\
r_{3}\\
~
\end{matrix},
\end{align}
where
\begin{align*}
r_{1}&=r(B,~C)-r(B),~r_{2}=r(B)+r(C)-r(B,~C),\\
r_{3}&=r(B,~C,~D)-r(B,~C),\\
r_{4}&=r(B,~D)+r(B,~C)-r(B)-r(B,~C,~D),\\
r_{5}+r_{6}&=r(B)+r(D)-r(B,~D),\\
r_{5}+r_{7}&=r(B,~C)+r(C,~D)-r(B,~C,~D)-r(C).
\end{align*}

\end{lemma}

Horn and Zhang \cite{FZhang3} presented an analogous special singular value decomposition for an $\eta$-Hermitian matrix.

\begin{lemma}
\label{lemma01}
(\cite{FZhang3})
Suppose that $A$ is $\eta$-Hermitian. Then there is a unitary matrix $U$ such that
\begin{align*}
UAU^{\eta*}=\begin{pmatrix}\Sigma&0\\0&0\end{pmatrix},
\end{align*}where $\Sigma=diag(\sigma_{1},\sigma_{2},\cdots,\sigma_{r})$ and
$\sigma_{1},\sigma_{2},\cdots,\sigma_{r}$ are  real positive  singular values of $A$.
\end{lemma}

Now we give the main theorem of this paper.

\begin{theorem}\label{theorem01}
Let $A=A^{\eta*}\in \mathbb{H}^{m\times m},
B\in \mathbb{H}^{m\times p_{1}},C\in \mathbb{H}^{m\times p_{2}},$ and $D\in \mathbb{H}^{m\times p_{3}}$ be given. Then there exist
 $P\in GL_{m}(\mathbb{H}), ~T_{1}\in GL_{p_{1}}(\mathbb{H}),~
T_{2}\in GL_{p_{2}}(\mathbb{H}),~T_{3}\in GL_{p_{3}}(\mathbb{H}),$ such that%
\begin{align}\label{4equ021}
PA P^{\eta*}=S_{A},\qquad PB T_{1}=S_{B},\qquad PC T_{2}=S_{C},\qquad PDT_{3}=S_{D},
\end{align}
where
\begin{align}\label{4equ022}
S_{A}=S_{A}^{\eta*}=\begin{pmatrix}
A_{11} & \cdots&   A_{19}&   A_{1,10}&0\\
\vdots& \ddots&   \vdots&\vdots&\vdots\\
A_{19}^{\eta*}& \cdots&A_{99}&A_{9,10}& 0\\
A_{1, 10}^{\eta*}& \cdots&A_{9,10}^{\eta*}&0&0\\0&\cdots&0 &0&\Sigma
\end{pmatrix},
\end{align}
\begin{align}\label{4equ023}
S_{B}=\begin{pmatrix}\begin{smallmatrix}
I_{m_{1}} &0&0&0&0&0\\
0 & I_{m_{2}}&0&0&0&0\\
0 & 0&I_{m_{3}}&0&0&0\\
0 & 0&0&I_{m_{4}}&0&0\\
0 & 0&0&0&I_{m_{5}}&0\\
0 & 0&0 &0&0&0\\
0 & 0&0 &0&0&0\\
0 & 0&0 &0&0&0\\
0 & 0&0 &0&0&0\\
0 & 0&0 &0&0&0\\
0 & 0&0 &0&0&0\end{smallmatrix}\end{pmatrix}
,
S_{C}=\begin{pmatrix}\begin{smallmatrix}
0 &0&0&I_{m_{1}}&0&0\\
0 & 0&0&0&I_{m_{2}}&0\\
0 & 0&0 &0&0&0\\
0 & 0&0 &0&0&0\\
0 & 0&0 &0&0&0\\
I_{m_{4}} & 0&0&0&0&0\\
0 & I_{m_{6}}&0&0&0&0\\
0 & 0&I_{m_{7}}&0&0&0\\
0 & 0&0 &0&0&0\\
0 & 0&0 &0&0&0\\
0 & 0&0 &0&0&0\end{smallmatrix}\end{pmatrix},
S_{D}=\begin{pmatrix}\begin{smallmatrix}
0 &0&0&0 &I_{m_{1}}&0\\
0 & 0&0 &0&0&0\\
0&0&0&I_{m_{3}}&0&0\\
0&I_{m_{4}} & 0&0&0&0\\
0 & 0&0 &0&0&0\\
0&I_{m_{4}} & 0&0&0&0\\
0 & 0&I_{m_{6}}&0&0&0\\
0 & 0&0 &0&0&0\\
I_{m_{8}}&0 & 0&0&0&0\\
0 & 0&0 &0&0&0\\
0 & 0&0 &0&0&0\end{smallmatrix}\end{pmatrix},
\end{align}
where $\Sigma$ is a diagonal and nonsingular matrix, and
\begin{align}
r(\Sigma)=r\begin{pmatrix}A&B&C&D\\B^{\eta*}&0&0&0\\C^{\eta*}&0&0&0\\D^{\eta*}&0&0&0\end{pmatrix}-2r(B,C,D),
\end{align}
\begin{align}\label{4equ024}
m_{1}=r(D)+r(B)+r(C)-r\begin{pmatrix}D&B&0\\D&0&C\end{pmatrix},
\end{align}
\begin{align}
m_{2}=r\begin{pmatrix}D&B&0\\D&0&C\end{pmatrix}-r(B,~C)-r(D),
~
m_{3}=r\begin{pmatrix}D&B&0\\D&0&C\end{pmatrix}-r(B,~D)-r(C),
\end{align}
\begin{align}
m_{4}=r(B,~C)+r(C,~D)+r(B,~D)-r(B,~C,~D)-r\begin{pmatrix}D&B&0\\D&0&C\end{pmatrix},
\end{align}
\begin{align}
m_{5}=r(B,~C,~D)-r(C,~D),
~
m_{6}=r\begin{pmatrix}D&B&0\\D&0&C\end{pmatrix}-r(C,~D)-r(B),
\end{align}
\begin{align}\label{4equ028}
m_{7}=r(B,~C,~D)-r(B,~D),
~
m_{8}=r(B,~C,~D)-r(B,~C),
\end{align}

\end{theorem}

\begin{proof}
It follows from Lemma \ref{lemma00} that there exist four   matrices
$P_{1}\in GL_{m}(\mathbb{H}),$ $W_{B}\in GL_{p_{1}}(\mathbb{H}),$ $
W_{C}\in GL_{p_{2}}(\mathbb{H}),$ and $W_{D}\in GL_{p_{3}}(\mathbb{H})$ such that
\begin{align*}
P_{1}(B,C,D)\begin{pmatrix}W_{B}&0&0\\0&W_{C}&0\\0&0&W_{D}\end{pmatrix}
=\end{align*}
\begin{align*}
\begin{pmatrix}\begin{matrix} I &0&0&0&0&0\\
0 & I&0&0&0&0\\
0 & 0&I&0&0&0\\
0 & 0&0&I&0&0\\
0 & 0&0&0&I&0\\
0 & 0&0 &0&0&0\\
0 & 0&0 &0&0&0\\
0 & 0&0 &0&0&0\\
0 & 0&0 &0&0&0\\
0 & 0&0 &0&0&0
\end{matrix}~&
\begin{matrix} 0 &0&0&I&0&0\\
0 & 0&0&0&I&0\\
0 & 0&0 &0&0&0\\
0 & 0&0 &0&0&0\\
0 & 0&0 &0&0&0\\
I & 0&0&0&0&0\\
0 & I&0&0&0&0\\
0 & 0&I&0&0&0\\
0 & 0&0 &0&0&0\\
0 & 0&0 &0&0&0
\end{matrix}~&
\begin{matrix} 0 &0&0&0 &I&0\\
0 & 0&0 &0&0&0\\
0&0&0&I&0&0\\
0&I & 0&0&0&0\\
0 & 0&0 &0&0&0\\
0&I & 0&0&0&0\\
0 & 0&I&0&0&0\\
0 & 0&0 &0&0&0\\
I&0 & 0&0&0&0\\
0 & 0&0 &0&0&0
\end{matrix}
\end{pmatrix}
\begin{matrix} m_{1}\\
m_{2}\\
m_{3}\\
m_{4}\\
m_{5}\\
m_{4}\\
m_{6}\\
m_{7}\\
m_{8}\\
m-r(B,C,D)
\end{matrix}.
\end{align*}
Let
\begin{align*}
P_{1}AP_{1}^{\eta*}=P_{1}A^{\eta*}P_{1}^{\eta*}\triangleq
\begin{pmatrix}
A_{11}^{(1)} & \cdots&   A_{1,10}^{(1)}\\
\vdots& \ddots&   \vdots\\
 A_{1,10}^{(1)\eta*} & \cdots&A_{10,10}^{(1)}
\end{pmatrix},
\end{align*}
where the symbol $\triangleq$ means ``equals by definition''.  Now we pay attention to the $\eta$-Hermitian matrix $A_{10,10}^{(1)}$. By Lemma \ref{lemma01}, we can find a unitary matrix $P_{2}$  such that
\begin{align*}
P_{2}A_{10,10}^{(1)}P_{2}^{\eta*}=\begin{pmatrix}0&0\\0&\Sigma\end{pmatrix},
\end{align*}where $\Sigma$ is a diagonal and nonsingular matrix, and $r(\Sigma)=r(A_{10,10}^{(1)}).$
Then we have
\begin{align*}
&\begin{pmatrix}I_{r(B,C,D)}&0\\0&P_{2}\end{pmatrix}\begin{pmatrix}
A_{11}^{(1)} & \cdots&   A_{1,10}^{(1)}\\
\vdots& \ddots&   \vdots\\
A_{1,10}^{(1)*}& \cdots&A_{10,10}^{(1)}
\end{pmatrix}\begin{pmatrix}I_{r(B,C,D)}&0\\0&P_{2}\end{pmatrix}^{\eta*}\\&\triangleq
\begin{pmatrix}
A_{11}^{(2)} & \cdots&   A_{19}^{(2)}&   A_{1,10}^{(2)}& A_{1,11}^{(2)}\\
\vdots& \ddots&   \vdots&\vdots&\vdots\\
A_{19}^{(2)\eta*} & \cdots&A_{99}^{(2)}&A_{9,10}^{(2)}& A_{9,11}^{(2)}\\
A_{1,10}^{(2)\eta*} & \cdots&A_{9,10}^{(2)\eta*}&0&0\\A_{1,11}^{(2)\eta*} &\cdots&A_{9,11}^{(2)\eta*} &0&\Sigma
\end{pmatrix},
\end{align*}
\begin{align*}
\begin{pmatrix} I_{r(B,C,D)}&0\\0&P_{2} \end{pmatrix}P_{1}(B,C,D)\begin{pmatrix}
 W_{B}&0&0\\0&W_{C}&0\\0&0&W_{D} \end{pmatrix}
=
\end{align*}
\begin{align*}
\begin{pmatrix}\begin{matrix} I &0&0&0&0&0\\
0 & I&0&0&0&0\\
0 & 0&I&0&0&0\\
0 & 0&0&I&0&0\\
0 & 0&0&0&I&0\\
0 & 0&0 &0&0&0\\
0 & 0&0 &0&0&0\\
0 & 0&0 &0&0&0\\
0 & 0&0 &0&0&0\\
0 & 0&0 &0&0&0\\
0 & 0&0 &0&0&0
\end{matrix}~&
\begin{matrix} 0 &0&0&I&0&0\\
0 & 0&0&0&I&0\\
0 & 0&0 &0&0&0\\
0 & 0&0 &0&0&0\\
0 & 0&0 &0&0&0\\
I & 0&0&0&0&0\\
0 & I&0&0&0&0\\
0 & 0&I&0&0&0\\
0 & 0&0 &0&0&0\\
0 & 0&0 &0&0&0\\
0 & 0&0 &0&0&0
\end{matrix}~&
\begin{matrix} 0 &0&0&0 &I&0\\
0 & 0&0 &0&0&0\\
0&0&0&I&0&0\\
0&I & 0&0&0&0\\
0 & 0&0 &0&0&0\\
0&I & 0&0&0&0\\
0 & 0&I&0&0&0\\
0 & 0&0 &0&0&0\\
I&0 & 0&0&0&0\\
0 & 0&0 &0&0&0\\
0 & 0&0 &0&0&0
\end{matrix}
\end{pmatrix}
\begin{matrix} m_{1}\\
m_{2}\\
m_{3}\\
m_{4}\\
m_{5}\\
m_{4}\\
m_{6}\\
m_{7}\\
m_{8}\\
m-r(B,C,D)-r(\Sigma)\\
r(\Sigma)
\end{matrix}.
\end{align*}
Let
\begin{align*}
P_{3}=\begin{pmatrix}I_{r_{bcd}}&\begin{pmatrix}0&-A_{1,11}^{(2)}\\
\vdots&\vdots\\0&-A_{9,11}^{(2)}\end{pmatrix}\\0&I_{m-r_{bcd}}\end{pmatrix}.
\end{align*}
Then we obtain
\begin{align*}
P_{3}\begin{pmatrix}
A_{11}^{(2)} & \cdots&   A_{19}^{(2)}&   A_{1,10}^{(2)}& A_{1,11}^{(2)}\\
\vdots& \ddots&   \vdots&\vdots&\vdots\\
A_{19}^{(2)\eta*} & \cdots&A_{99}^{(2)}&A_{9,10}^{(2)}& A_{9,11}^{(2)}\\
A_{1,10}^{(2)\eta*} & \cdots&A_{9,10}^{(2)\eta*}&0&0\\A_{1,11}^{(2)\eta*} &\cdots&A_{9,11}^{(2)\eta*} &0&\Sigma
\end{pmatrix}P_{3}^{\eta*}\triangleq
\begin{pmatrix}
A_{11} & \cdots&   A_{19}&   A_{1,10}&0\\
\vdots& \ddots&   \vdots&\vdots&\vdots\\
A_{19}^{\eta*}& \cdots&A_{99}&A_{9,10}& 0\\
A_{1,10}^{\eta*}& \cdots&A_{9,10}^{\eta*}&0&0\\0 &\cdots&0 &0&\Sigma
\end{pmatrix}.
\end{align*}
Let
\begin{align*}
P\triangleq P_{3}\begin{pmatrix}I_{r(B,C,D)}&0\\0&P_{2}\end{pmatrix}P_{1},
~
T_{1}=W_{C},~T_{2}=W_{D},~T_{3}=W_{E}.
\end{align*}
Hence,  the matrices  $P\in GL_{m}(\mathbb{H}),~T_{1}\in GL_{p_{1}}(\mathbb{H}),~
T_{2}\in GL_{p_{2}}(\mathbb{H}),$ and $T_{3}\in GL_{p_{3}}(\mathbb{H})$ satisfy the equation (\ref{4equ021}). Now we want to give the dimensions of
$r(\Sigma),m_{1},\ldots,m_{8}.$ It is easy to verify that
\begin{align*}
r(\Sigma)=r\begin{pmatrix}A&B&C&D\\B^{\eta*}&0&0&0\\C^{\eta*}&0&0&0\\D^{\eta*}&0&0&0\end{pmatrix}-2r(B,C,D).
\end{align*}It follows from $S_{A},S_{B},S_{C},$ and $S_{D}$ in
(\ref{4equ022})-(\ref{4equ023}) that
\begin{align*}
\begin{pmatrix}
1&1&1&1&1&0&0&0\\
1&1&0&1&0&1&1&0\\
1&0&1&1&0&1&0&1\\
1&1&1&2&1&1&1&0\\
1&1&1&2&1&1&0&1\\
1&1&1&2&0&1&1&1\\
1&1&1&2&1&1&1&1\\
1&0&1&1&0&1&0&1\end{pmatrix}\begin{pmatrix}m_{1}\\
m_{2}\\
m_{3}\\
m_{4}\\
m_{5}\\
m_{6}\\
m_{7}\\
m_{8}\end{pmatrix}=\begin{pmatrix}r(B)\\
r(C)\\
r(D)\\
r(B,C)\\
r(B,D)\\
r(C,D)\\
r(B,C,D)\\
r\begin{pmatrix}\begin{smallmatrix}D&B&0\\D&0&C\end{smallmatrix}\end{pmatrix}-r(B)-r(C)\end{pmatrix}.
\end{align*}
Solving for $m_{i},(i=1,\ldots,8)$ gives (\ref{4equ024})-(\ref{4equ028}).
\end{proof}

Let $D$ vanish in Theorem \ref{theorem01}, then we obtain the simultaneous decomposition of a matrix triplet with the same row numbers
\begin{align*}
(A,~B,~C),
\end{align*}
where $A$ is an $\eta$-Hermitian matrix.

\begin{corollary}
Let $A=A^{\eta*}\in \mathbb{H}^{m\times m},
B\in \mathbb{H}^{m\times p_{1}},$ and $C\in \mathbb{H}^{m\times p_{2}}$ be given. Then there exist
 $P\in GL_{m}(\mathbb{H}), ~T_{1}\in GL_{p_{1}}(\mathbb{H}),~
T_{2}\in GL_{p_{2}}(\mathbb{H}),$ such that%
\begin{align*}
PA P^{\eta*}=S_{A},\qquad PB T_{1}=S_{B},\qquad PC T_{2}=S_{C},
\end{align*}
where
\begin{align*}
(S_{A},~S_{B},~S_{C})=
\bordermatrix{
~& \cr
  \begin{matrix}n_{1}\\n_{2}\\n_{3}\\ \\n_{4}\end{matrix}&
  \begin{matrix}
A_{11}^{1}&A_{12}^{1}&A_{13}^{1}&A_{14}^{1}&0\\
(A_{12}^{1})^{\eta*}&A_{22}^{1}&A_{23}^{1}&A_{24}^{1}&0\\
(A_{13}^{1})^{\eta*}&(A_{23}^{1})^{\eta*}&A_{33}^{1}&A_{34}^{1}&0\\
(A_{14}^{1})^{\eta*}&(A_{24}^{1})^{\eta*}&(A_{34}^{1})^{\eta*}&0&0\\
0&0&0&0&\Sigma_{1}
\end{matrix}~&~
\begin{matrix}
I&0&0\\
0&I&0\\
0&0&0\\
0&0&0\\
0&0&0
\end{matrix}~&~
\begin{matrix}
I&0&0\\
0&0&0\\
0&I&0\\
0&0&0\\
0&0&0
\end{matrix}
},
\end{align*}
where
\begin{align*}
n_{1}=r(B)+r(C)-r(B,~C),~n_{2}=r(B,~C)-r(C),~n_{3}=r(B,~C)-r(B),
\end{align*}
\begin{align*}
n_{4}=r\begin{pmatrix}A&B&C\\B^{\eta*}&0&0\\C^{\eta*}&0&0\end{pmatrix}-2r(B,~C).
\end{align*}

\end{corollary}

\section{\textbf{Solvability conditions and general $\eta$-Hermitian solution to (\ref{4system001})}}

In this section, we give some solvability conditions for the real quaternion matrix equation (\ref{4system001}) to possess an $\eta$-Hermitian
solution and to present an expression of this $\eta$-Hermitian solution when the solvability
conditions are met.   A numerical example is also given to illustrate the main result.

\begin{theorem}\label{theorem04}
Let $A=A^{\eta*}\in \mathbb{H}^{m\times m},
B\in \mathbb{H}^{m\times p_{1}},C\in \mathbb{H}^{m\times p_{2}},$ and $D\in \mathbb{H}^{m\times p_{3}}$ be given.  Then the real quaternion matrix equation (\ref{4system001}) has an $\eta$-Hermitian solution $(X,Y,Z)$ if and only if the ranks satisfy:
\begin{align}\label{rank01}
r(A,B,C,D)=r(B,C,D),~r\begin{pmatrix}A&B&C\\D^{\eta*}&0&0\end{pmatrix}=r(B,C)+r(D),
\end{align}
\begin{align}
r\begin{pmatrix}A&B&D\\C^{\eta*}&0&0\end{pmatrix}=r(B,D)+r(C),
~r\begin{pmatrix}A&C&D\\B^{\eta*}&0&0\end{pmatrix}=r(C,D)+r(B),
\end{align}
\begin{align}\label{rank02}
r\begin{pmatrix}0&D^{\eta*}&D^{\eta*}&0&0\\
D&-A&0&0&B\\
D&0&A&C&0\\
0&C^{\eta*}&0&0&0\\
0&0&B^{\eta*}&0&0\end{pmatrix}=2r\begin{pmatrix}D&B&0\\D&0&C\end{pmatrix}.
\end{align}
In this case, the general $\eta$-Hermitian solution to (\ref{4system001}) can be expressed as
\begin{align*}
X=T_{1}\widehat{X}T_{1}^{\eta*},\quad Y=T_{2}\widehat{Y}T_{2}^{\eta*},\quad Z=T_{3}\widehat{Z}T_{3}^{\eta*},
\end{align*}
where
\begin{align}\label{equh033}
\widehat{X}=\widehat{X}^{\eta*}=\bordermatrix{
& m_{1}&m_{2} & m_{3}&m_{4}&m_{5}&p_{1}-r(B) \cr
&X_{11}&X_{12}&X_{13}&X_{14}&A_{15}&X_{16} \cr
&X_{12}^{\eta*}&X_{22}&A_{23}&A_{24}&A_{25}&X_{26} \cr
&X_{13}^{\eta*}&A_{23}^{\eta*}&X_{33}&A_{34}-A_{36}&A_{35}&X_{36} \cr
&X_{14}^{\eta*}&A_{24}^{\eta*}&(A_{34}-A_{36})^{\eta*}&A_{44}-A_{46}&A_{45}&X_{46} \cr
&A_{15}^{\eta*}&A_{25}^{\eta*}&A_{35}^{\eta*}&A_{45}^{\eta*}&A_{55}&X_{56} \cr
&X_{16}^{\eta*}&X_{26}^{\eta*}&X_{36}^{\eta*}&X_{46}^{\eta*}&X_{56}^{\eta*}&X_{66}},
\end{align}
\begin{align}
\widehat{Y}=\widehat{Y}^{\eta*}=\bordermatrix{
& m_{4}&m_{6} & m_{7}&m_{1}&m_{2}&p_{2}-r(C) \cr
&A_{66}-A_{46}&A_{67}-A_{47}&A_{68}&A_{16}^{\eta*}-A_{14}^{\eta*}+X_{14}^{\eta*}&A_{26}^{\eta*}&Y_{16} \cr
&(A_{67}-A_{47})^{\eta*}&Y_{22}&A_{78}&Y_{24}&A_{27}^{\eta*}&Y_{26} \cr
&A_{68}^{\eta*}&A_{78}^{\eta*}&A_{88}&A_{18}^{\eta*}&A_{28}^{\eta*}&Y_{36} \cr
&A_{16}-A_{14}+X_{14}&Y_{24}^{\eta*}&A_{18}&Y_{44}&A_{12}-X_{12}&Y_{46} \cr
&A_{26}&A_{27}&A_{28}&(A_{12}-X_{12})^{\eta*}&A_{22}-X_{22}&Y_{56} \cr
&Y_{16}^{\eta*}&Y_{26}^{\eta*}&Y_{36}^{\eta*}&Y_{46}^{\eta*}&Y_{56}^{\eta*}&Y_{66}},
\end{align}
\begin{align}\label{equhh035}
\widehat{Z}=\widehat{Z}^{\eta*}=\bordermatrix{
& m_{8}&m_{4} & m_{6}&m_{3}&m_{1}&p_{3}-r(D) \cr
&A_{99}&A_{69}^{\eta*}&A_{79}^{\eta*}&A_{39}^{\eta*}&A_{19}^{\eta*}&Z_{16} \cr
&A_{69}&A_{46}&A_{47}&A_{36}^{\eta*}&(A_{14}-X_{14})^{\eta*}&Z_{26} \cr
&A_{79}&A_{47}^{\eta*}&A_{77}-Y_{22}&A_{37}^{\eta*}&A_{17}^{\eta*}-Y_{24}&Z_{36} \cr
&A_{39}&A_{36}&A_{37}&A_{33}-X_{33}&(A_{13}-X_{13})^{\eta*}&Z_{46} \cr
&A_{19}&A_{14}-X_{14}&A_{17}-Y_{24}^{\eta*}&A_{13}-X_{13}&Z_{55}&Z_{56} \cr
&Z_{16}^{\eta*}&Z_{26}^{\eta*}&Z_{36}^{\eta*}&Z_{46}^{\eta*}&Z_{56}^{\eta*}&Z_{66}},
\end{align}
in which $X_{11},X_{22},X_{33},X_{66},Y_{22},Y_{44},Y_{66},Z_{55},$ and $Z_{66}$  are arbitrary
$\eta$-Hermitian matrices over $\mathbb{H}$ with appropriate sizes,  the remaining $X_{ij},Y_{ij},Z_{ij}$  are arbitrary matrices over $\mathbb{H}$
with appropriate sizes.
\end{theorem}

\begin{proof}
 Observe that the dimensions of the coefficient matrices $A,B,C,$ and $D$ in the real quaternion matrix equation (\ref{4system001}) have the same number of rows. Hence, the coefficient matrices $A,B,C,D$ can be arranged in the following matrix array
\begin{align*}
\begin{pmatrix}A&B&C&D\end{pmatrix}.
\end{align*}It follows from Theorem \ref{theorem01} that there exist
 $P\in GL_{m}(\mathbb{H}), ~T_{1}\in GL_{p_{1}}(\mathbb{H}),~
T_{2}\in GL_{p_{2}}(\mathbb{H}),~T_{3}\in GL_{p_{3}}(\mathbb{H}),$ such that%
\begin{align*}
PA P^{\eta*}=S_{A},\qquad PB T_{1}=S_{B},\qquad PC T_{2}=S_{C},\qquad PDT_{3}=S_{D},
\end{align*}where $S_{A},S_{B},S_{C},$ and $S_{D}$ are given in (\ref{4equ022}) and (\ref{4equ023}). Hence the matrix equation (\ref{4system001}) is equivalent to the matrix equation
\begin{align*}
P^{-1}S_{B}(T_{1}XT_{1}^{\eta*})S_{B}^{\eta*}P^{-\eta*}+P^{-1}S_{C}(T_{2}YT_{2}^{\eta*})S_{C}^{\eta*}P^{-\eta*}
+P^{-1}S_{D}(T_{3}ZT_{3}^{\eta*})S_{D}^{\eta*}P^{-\eta*}=P^{-1}S_{A}P^{-\eta*},
\end{align*}
i.e.,
\begin{align}\label{equh035}
S_{B}(T_{1}XT_{1}^{\eta*})S_{B}^{\eta*}+S_{C}(T_{2}YT_{2}^{\eta*})S_{C}^{\eta*}+S_{D}(T_{3}ZT_{3}^{\eta*})S_{D}^{\eta*}=S_{A}.
\end{align}
Let the matrices
\begin{align}\label{equh036}
\widehat{X}=T_{1}^{-1}XT_{1}^{-\eta*}=\begin{pmatrix}X_{11}&\cdots&X_{16}\\
\vdots&\ddots&\vdots\\
X_{16}^{\eta*}&\cdots&X_{66}\end{pmatrix}=\widehat{X}^{\eta*},
\end{align}
\begin{align}
\widehat{Y}=T_{2}^{-1}YT_{2}^{-\eta*}=\begin{pmatrix}Y_{11}&\cdots&Y_{16}\\
\vdots&\ddots&\vdots\\
Y_{16}^{\eta*}&\cdots&Y_{66}\end{pmatrix}=\widehat{Y}^{\eta*},
\end{align}
\begin{align}\label{equh038}
\widehat{Z}=T_{3}^{-1}ZT_{3}^{-\eta*}=\begin{pmatrix}Z_{11}&\cdots&Z_{16}\\
\vdots&\ddots&\vdots\\
Z_{61}&\cdots&Z_{66}\end{pmatrix}=\widehat{Z}^{\eta*},
\end{align}
be partitioned in accordance with (\ref{equh035}). Substituting $\widehat{X},\widehat{Y},$ and $\widehat{Z}$ of (\ref{equh036})-(\ref{equh038}) into (\ref{equh035}) yields
\begin{align*}
\begin{pmatrix}
\scriptstyle  X_{11}+Y_{44}+Z_{55}&\scriptstyle   X_{12}+Y_{45}&\scriptstyle   X_{13}+Z_{45}^{\eta*}&\scriptstyle   X_{14}+Z_{25}^{\eta*}&\scriptstyle   X_{15} &\scriptstyle    Y_{14}^{\eta*}+Z_{25}^{\eta*} &\scriptstyle   Y_{24}^{\eta*}+Z_{35}^{\eta*}&\scriptstyle    Y_{34}^{\eta*}&\scriptstyle    Z_{15}^{\eta*} &\scriptstyle      0 &\scriptstyle    0 \\
\scriptstyle  X_{12}^{\eta*}+Y_{45}^{\eta*}&\scriptstyle   X_{22}+Y_{55}&\scriptstyle   X_{23}&\scriptstyle   X_{24}&\scriptstyle   X_{25} &\scriptstyle    Y_{15}^{\eta*} &\scriptstyle   Y_{25}^{\eta*}&\scriptstyle    Y_{35}^{\eta*} &\scriptstyle    0 &\scriptstyle      0 &\scriptstyle    0 \\
\scriptstyle   X_{13}^{\eta*}+Z_{45}&\scriptstyle   X_{23}^{\eta*}&\scriptstyle   X_{33}+Z_{44}&\scriptstyle   X_{34}+Z_{24}^{\eta*}&\scriptstyle   X_{35} &\scriptstyle    Z_{24}^{\eta*} &\scriptstyle   Z_{34}^{\eta*} &\scriptstyle   0 &\scriptstyle    Z_{14}^{\eta*} &\scriptstyle     0&\scriptstyle    0 \\
\scriptstyle  X_{14}^{\eta*}+Z_{25}&\scriptstyle   X_{24}^{\eta*}&\scriptstyle   X_{34}^{\eta*}+Z_{24}&\scriptstyle   X_{44}+Z_{22}&\scriptstyle   X_{45} &\scriptstyle    Z_{22}&\scriptstyle   Z_{23} &\scriptstyle   0&\scriptstyle    Z_{12}^{\eta*} &\scriptstyle     0 &\scriptstyle    0 \\
\scriptstyle  X_{15}^{\eta*}&\scriptstyle   X_{25}^{\eta*}&\scriptstyle   X_{35}^{\eta*}&\scriptstyle   X_{45}^{\eta*}&\scriptstyle   X_{55} &\scriptstyle    0 &\scriptstyle   0 &\scriptstyle    0 &\scriptstyle    0 &\scriptstyle      0 &\scriptstyle    0\\
\scriptstyle  Y_{14}+Z_{25} &\scriptstyle    Y_{15} &\scriptstyle    Z_{24} &\scriptstyle   Z_{22} &\scriptstyle    0 &\scriptstyle    Y_{11} +Z_{22}&\scriptstyle   Y_{12} +Z_{23} &\scriptstyle   Y_{13}&\scriptstyle   Z_{12}^{\eta*} &\scriptstyle      0 &\scriptstyle    0 \\
\scriptstyle  Y_{24}+Z_{35} &\scriptstyle    Y_{25} &\scriptstyle    Z_{34} &\scriptstyle   Z_{23}^{\eta*} &\scriptstyle    0 &\scriptstyle    Y_{12}^{\eta*} +Z_{23}^{\eta*}&\scriptstyle   Y_{22}+Z_{33} &\scriptstyle   Y_{23} &\scriptstyle     Z_{13}^{\eta*}&\scriptstyle     0 &\scriptstyle    0 \\
\scriptstyle  Y_{34} &\scriptstyle   Y_{35}&\scriptstyle    0 &\scriptstyle   0 &\scriptstyle    0 &\scriptstyle    Y_{13}^{\eta*} &\scriptstyle   Y_{23}^{\eta*} &\scriptstyle   Y_{33} &\scriptstyle    0 &\scriptstyle     0 &\scriptstyle    0\\
\scriptstyle   Z_{15} &\scriptstyle    0&\scriptstyle    Z_{14} &\scriptstyle   Z_{12} &\scriptstyle    0 &\scriptstyle    Z_{12} &\scriptstyle   Z_{13} &\scriptstyle    0 &\scriptstyle    Z_{11} &\scriptstyle      0 &\scriptstyle    0 \\
\scriptstyle   0 &\scriptstyle    0 &\scriptstyle   0 &\scriptstyle   0 &\scriptstyle   0 &\scriptstyle    0 &\scriptstyle   0 &\scriptstyle    0 &\scriptstyle   0 &\scriptstyle      0 &\scriptstyle   0\\
\scriptstyle  0 &\scriptstyle    0 &\scriptstyle    0 &\scriptstyle    0 &\scriptstyle      0 &\scriptstyle    0 &\scriptstyle    0&\scriptstyle    0 &\scriptstyle    0 &\scriptstyle    0 &\scriptstyle     0
\end{pmatrix}
\end{align*}
\begin{align}\label{equ0048}
=\begin{pmatrix}
A_{11} & \cdots&   A_{19}&   A_{1,10}&0\\
\vdots& \ddots&   \vdots&\vdots&\vdots\\
A_{19}^{\eta*}& \cdots&A_{99}&A_{9,10}& 0\\
A_{1, 10}^{\eta*}& \cdots&A_{9,10}^{\eta*}&0&0\\0&\cdots&0 &0&\Sigma
\end{pmatrix}.
\end{align}

If the equation (\ref{4system001}) has an $\eta$-Hermitian solution $(X,Y,Z)$, then by (\ref{equ0048}), we obtain that
\begin{align}\label{equh031}
\Sigma=0,~A_{49}=A_{69},~A_{46}=A_{46}^{\eta*},
~ \left(A_{1,10}^{\eta*},~  \cdots, ~A_{9,10}^{\eta*}\right)=0,
\end{align}
\begin{align}\label{equh032}
A_{29}=0, ~A_{38}=0, ~A_{48}=0, ~A_{56}=0,~A_{57}=0,~A_{58}=0,~A_{59}=0,~A_{89}=0.
\end{align}
and
\begin{align*}
&X_{11}+Y_{44}+Z_{55}=A_{11},~X_{12}+Y_{45}=A_{12},~X_{13}+Z_{54}=A_{13},~X_{14}+Z_{52}=A_{14},~X_{15}=A_{15}, \\
&Y_{41}+Z_{52}=A_{16},~Y_{42}+Z_{53}=A_{17},~ Y_{43}=A_{18},~ Z_{51}=A_{19},X_{21}+Y_{54}=A_{21},~X_{22}+Y_{55}=A_{22},\\
&X_{23}=A_{23},~X_{24}=A_{24},~X_{25}=A_{25},~Y_{51}=A_{26},~Y_{52}=A_{27},~Y_{53}=A_{28},~X_{31}+Z_{45}=A_{31}, \\
&X_{32}=A_{32}, ~X_{33}+Z_{44}=A_{33}, ~X_{34}+Z_{42}=A_{34}, ~X_{35}=A_{35}, ~ Z_{42}=A_{36}, ~Z_{43}=A_{37}, ~ Z_{41}=A_{39},\\
&X_{41}+Z_{25}=A_{41}, ~X_{42}=A_{42}, ~X_{43}+Z_{24}=A_{43}, ~X_{44}+Z_{22}=A_{44}, ~X_{45}=A_{45}, ~ Z_{22}=A_{46}, \\
&Z_{23}=A_{47},  ~ Z_{21}=A_{49}, ~X_{51}=A_{51}, ~X_{52}=A_{52}, ~X_{53}=A_{53}, ~X_{54}=A_{54}, ~X_{55}=A_{55}\\
&Y_{14}+Z_{25}=A_{61}, ~ Y_{15}=A_{62}, ~ Z_{24}=A_{63}, ~Z_{22}=A_{64},  ~ Y_{11}+Z_{22}=A_{66}, ~Y_{12}+Z_{23}=A_{67}, \\
& Y_{13}=A_{68}, ~Z_{21}=A_{69},  ~Y_{24}+Z_{35}=A_{71}, ~ Y_{25}=A_{72}, ~ Z_{34}=A_{73}, ~Z_{32}=A_{74},  ~ Y_{21} +Z_{32}=A_{76}, \\
& Y_{22}+Z_{33}=A_{77}, ~Y_{23}=A_{78} , ~  Z_{31}=A_{79},  ~Y_{34}=A_{81}, ~Y_{35}=A_{82}, ~ Y_{31}=A_{86}, ~Y_{32}=A_{87}, \\
& Y_{33}=A_{88},~Z_{15}=A_{91},  ~ Z_{14}=A_{93}, ~Z_{12}=A_{94},  ~ Z_{12}=A_{96}, ~Z_{13}=A_{97},   ~ Z_{11}=A_{99}.
\end{align*}
Hence, the general $\eta$-Hermitian solution $(X,Y,Z)$ can be expressed as (\ref{equh033})-(\ref{equhh035}) by (\ref{equ0048}).

Conversely, assume that the equalities in (\ref{equh031}) and (\ref{equh032}) hold, then by (\ref{equh036})-(\ref{equ0048}), it can be
verified that the matrices have the forms of (\ref{equh033})-(\ref{equhh035}) is an $\eta$-Hermitian solution of (\ref{equh035}), i.e., (\ref{4system001}).

We now show that (\ref{rank01})-(\ref{rank02}) $\Longleftrightarrow$ (\ref{equh031}) and (\ref{equh032}). From $S_{A},S_{B},S_{C},$ and $S_{D}$ in Theorem \ref{theorem01}, we can infer that
\begin{align*}
r(A,~B,~C,~D)=r(B,~C,~D) \Longleftrightarrow \left(A_{1,10}^{\eta*},~  \cdots, ~A_{9,10}^{\eta*}\right)=0,~\Sigma=0,
\end{align*}
\begin{align*}
r\begin{pmatrix}A&B&C\\D^{\eta*}&0&0\end{pmatrix}=r(B,C)+r(D)\Longleftrightarrow A_{29}=0,~A_{89}=0,~A_{49}=A_{69},~\Sigma=0,
\end{align*}
\begin{align*}
r\begin{pmatrix}A&B&D\\C^{\eta*}&0&0\end{pmatrix}=r(B,D)+r(C)\Longleftrightarrow A_{38}=0,~A_{48}=0,~A_{58}=0,~A_{89}=0,~\Sigma=0,
\end{align*}
\begin{align*}
~r\begin{pmatrix}A&C&D\\B^{\eta*}&0&0\end{pmatrix}=r(C,D)+r(B)\Longleftrightarrow A_{56}=0,~A_{57}=0,~A_{58}=0,~A_{59}=0,~\Sigma=0,
\end{align*}
 \begin{align*}
r\begin{pmatrix}0&D^{\eta*}&D^{\eta*}&0&0\\
D&-A&0&0&B\\
D&0&A&C&0\\
0&C^{\eta*}&0&0&0\\
0&0&B^{\eta*}&0&0\end{pmatrix}=2r\begin{pmatrix}D&B&0\\D&0&C\end{pmatrix}
\Longleftrightarrow  A_{46}=A_{46}^{\eta*},~\Sigma=0.
\end{align*}

\end{proof}

Now we give an example to illustrate Theorem \ref{theorem04}.

\begin{example}
Given the real quaternion matrices:
\begin{align*}
B=\begin{pmatrix}\mathbf{i}+\mathbf{j}+\mathbf{k}&1&1+\mathbf{i}+\mathbf{j}-\mathbf{k}\\
-1-\mathbf{j}+\mathbf{k}&\mathbf{i}&-1+\mathbf{i}+\mathbf{j}+\mathbf{k}\end{pmatrix},
C=\begin{pmatrix}1&2\mathbf{i}+\mathbf{j}&-1+\mathbf{k}\\ \mathbf{i}+\mathbf{k}&1+\mathbf{i}+\mathbf{j}-\mathbf{k}&0\end{pmatrix},
\end{align*}
\begin{align*}
D=\begin{pmatrix}\mathbf{j}+2\mathbf{k}&\mathbf{i}+\mathbf{k}&\mathbf{j}\\ -2\mathbf{j}+\mathbf{k}&-1-\mathbf{j}&\mathbf{k}\end{pmatrix},
A=A^{\mathbf{j}*}=\begin{pmatrix}-1+5\mathbf{i}-20\mathbf{k}&-25-2\mathbf{i}-17\mathbf{j}-5\mathbf{k}\\
-25-2\mathbf{i}+17\mathbf{j}-5\mathbf{k}&-9-18\mathbf{i}-14\mathbf{k}\end{pmatrix}.
\end{align*}
Now we consider the $\mathbf{j}$-Hermitian solution to the real quaternion matrix equation (\ref{4system001}). Check that
\begin{align*}
r(A,B,C,D)=r(B,C,D)=2,
\end{align*}
\begin{align*}
r\begin{pmatrix}A&B&C\\D^{\eta*}&0&0\end{pmatrix}=r(B,C)+r(D)=3,
\end{align*}
\begin{align*}
r\begin{pmatrix}A&B&D\\C^{\eta*}&0&0\end{pmatrix}=r(B,D)+r(C)=3,
\end{align*}
\begin{align*}
r\begin{pmatrix}A&C&D\\B^{\eta*}&0&0\end{pmatrix}=r(C,D)+r(B)=3,
\end{align*}
\begin{align*}
r\begin{pmatrix}0&D^{\eta*}&D^{\eta*}&0&0\\
D&-A&0&0&B\\
D&0&A&C&0\\
0&C^{\eta*}&0&0&0\\
0&0&B^{\eta*}&0&0\end{pmatrix}=2r\begin{pmatrix}D&B&0\\D&0&C\end{pmatrix}=6.
\end{align*}
All the rank equalities in (\ref{rank01})-(\ref{rank02}) hold. Hence, the real quaternion matrix equation (\ref{4system001}) has a $\mathbf{j}$-Hermitian solution $(X,Y,Z)$. Note that
\begin{align*}
X=X^{\mathbf{j}*}=\begin{pmatrix}1&\mathbf{i}+\mathbf{k}&0\\ \mathbf{i}+\mathbf{k}&1+\mathbf{i}&1-\mathbf{k}\\0&1-\mathbf{k}&0\end{pmatrix},
~
Y=Y^{\mathbf{j}*}=\begin{pmatrix}
0&1+\mathbf{i}&\mathbf{k}\\1+\mathbf{i}&\mathbf{i}&2\mathbf{k}\\ \mathbf{k}&2\mathbf{k}&1
\end{pmatrix},
\end{align*}
\begin{align*}
Z=Z^{\mathbf{j}*}=\begin{pmatrix}\mathbf{i}&\mathbf{i}-\mathbf{k}&\mathbf{k}\\
\mathbf{i}-\mathbf{k}&\mathbf{i}&1\\
\mathbf{k}&1&1\end{pmatrix}
\end{align*}
satisfy the real quaternion matrix equation (\ref{4system001}).
\end{example}

\section{\textbf{The solution to (\ref{4system002}) with $Y$ being $\eta$-Hermitian}}

In this section, we consider the real quaternion matrix equation (\ref{4system002}). We derive necessary and sufficient conditions for (\ref{4system002}) in terms of ranks of the coefficient matrices. We also give the general solution to this  real quaternion matrix equation. A numerical example is also given to illustrate the main result.

\begin{theorem}\label{theorem05}
Let $A=A^{\eta*}\in \mathbb{H}^{m\times m},
B\in \mathbb{H}^{m\times p_{1}},C\in \mathbb{H}^{p_{2}\times m},$ and $D\in \mathbb{H}^{m\times p_{3}}$ be given.  Then the real quaternion matrix equation (\ref{4system002}) has a solution $(X,Y)$, where $Y$ is $\eta$-Hermitian, if and only if the ranks satisfy:
\begin{align}\label{rank03}
r(A,B,C^{\eta*},D)=r(B,C^{\eta*},D),
\end{align}
\begin{align}
r\begin{pmatrix}A&B&C^{\eta*}\\D^{\eta*}&0&0\end{pmatrix}=r(B,C^{\eta*})+r(D),
\end{align}
\begin{align}
r\begin{pmatrix}A&B&D\\B^{\eta*}&0&0\end{pmatrix}=r(B,D)+r(B),
\end{align}
\begin{align}
r\begin{pmatrix}A&C^{\eta*}&D\\C&0&0\end{pmatrix}=r(C^{\eta*},D)+r(C),
\end{align}
\begin{align}\label{rank04}
r\begin{pmatrix}
A&0&B&0&D\\
0&-A&0&C^{\eta*}&D\\
B^{\eta*}&0&0&0&0\\
0&C&0&0&0\\
D^{\eta*}&D^{\eta*}&0&0&0\end{pmatrix}=2r\begin{pmatrix}B&0&D\\0&C^{\eta*}&D\end{pmatrix}.
\end{align}
In this case, the general solution to (\ref{4system002}) can be expressed as
\begin{align*}
X=T_{1}\widehat{X}T_{2}^{\eta*},\quad Y=T_{3}\widehat{Y}T_{3}^{\eta*},
\end{align*}
where
\begin{align}\label{hhequX}
\widehat{X}=
\begin{pmatrix}
X_{11}&X_{12}&A_{18}&X_{14}&A_{12}-X_{24}^{\eta*}&X_{16}\\
A_{26}&A_{27}&A_{28}&X_{24}&\frac{1}{2}A_{22}+Z&X_{26}\\
A_{36}-A_{34}&X_{32}&A_{38}&X_{34}&A_{23}^{\eta*}&X_{36}\\
A_{46}-A_{44}&A_{47}-A_{67}&A_{48}&A_{14}^{\eta*}-A_{16}^{\eta*}+X_{11}^{\eta*}&A_{24}^{\eta*}&X_{46}\\
A_{56}&A_{57}&A_{58}&A_{15}^{\eta*}&A_{25}^{\eta*}&X_{56}\\
X_{61}&X_{62}&X_{63}&X_{64}&X_{65}&X_{66}
\end{pmatrix}
\end{align}
\begin{align}\label{hhequY}
\widehat{Y}=
\begin{pmatrix}
A_{99}&A_{49}^{\eta*}&A_{79}^{\eta*}&A_{39}^{\eta*}&A_{19}^{\eta*}&Y_{16}\\
A_{49}&A_{44}&A_{67}&A_{34}^{\eta*}&A_{14}^{\eta*}-X_{44}&Y_{26}\\
A_{79}&A_{67}^{\eta*}&A_{77}&A_{37}^{\eta*}-X_{32}^{\eta*}&A_{17}^{\eta*}-X_{12}^{\eta*}&Y_{36}\\
A_{39}&A_{34}&A_{37}-X_{32}&A_{33}&A_{13}^{\eta*}-X_{34}&Y_{46}\\
A_{19}&A_{14}-X_{44}^{\eta*}&A_{17}-X_{12}&A_{13}-X_{34}^{\eta*}&A_{11}-X_{14}-X_{14}^{\eta*}&Y_{56}\\
Y_{16}^{\eta*}&Y_{26}^{\eta*}&Y_{36}^{\eta*}&Y_{46}^{\eta*}&Y_{56}^{\eta*}&Y_{66}
\end{pmatrix}
\end{align}
in which $Y_{66}$ and $Z$  are arbitrary
$\eta$-Hermitian matrices and skew-$\eta$-Hermitian matrices over $\mathbb{H}$ with appropriate sizes,  the remaining $X_{ij}$ and $Y_{ij}$  are arbitrary matrices over $\mathbb{H}$
with appropriate sizes.

\end{theorem}

\begin{proof}
 Note that the dimensions of the coefficient matrices $A,B,C^{\eta*},$ and $D$ in real quaternion matrix equation (\ref{4system002}) have the same number of rows. Hence, the coefficient matrices $A,B,C,D$ can be arranged in the following matrix array
\begin{align*}
\begin{pmatrix}A&B&C^{\eta*}&D\end{pmatrix}.
\end{align*}It follows from Theorem \ref{theorem01} that there exist
 $P\in GL_{m}(\mathbb{H}), ~T_{1}\in GL_{p_{1}}(\mathbb{H}),~
T_{2}\in GL_{p_{2}}(\mathbb{H}),~T_{3}\in GL_{p_{3}}(\mathbb{H}),$ such that%
\begin{align*}
PA P^{\eta*}=S_{A},\qquad PB T_{1}=S_{B},\qquad PC^{\eta*} T_{2}=S_{C},\qquad PDT_{3}=S_{D},
\end{align*}where $S_{A},S_{B},S_{C},$ and $S_{D}$ are given in (\ref{4equ022}) and (\ref{4equ023}). Hence the real quaternion matrix equation (\ref{4system002}) is equivalent to the real quaternion matrix equation
\begin{align*}
&P^{-1}S_{B}(T_{1}^{-1}XT_{2}^{-\eta*})S_{C}^{\eta*}P^{-\eta*}+P^{-1}S_{C}(T_{2}^{-1}X^{\eta*}T_{1}^{-\eta*})S_{B}^{\eta*}P^{-\eta*}
+P^{-1}S_{D}(T_{3}YT_{3}^{\eta*})S_{D}^{\eta*}P^{-\eta*}\\
&=P^{-1}S_{A}P^{-\eta*},
\end{align*}
i.e.,
\begin{align}\label{hhequ041}
S_{B}(T_{1}^{-1}XT_{2}^{-\eta*})S_{C}^{\eta*}+S_{C}(T_{2}^{-1}X^{\eta*}T_{1}^{-\eta*})S_{B}^{\eta*}
+S_{D}(T_{3}YT_{3}^{\eta*})S_{D}^{\eta*}=S_{A}.
\end{align}
Let the matrices
\begin{align} \label{hhequ042}
\widehat{X}=T_{1}^{-1}XT_{2}^{-\eta*}=\begin{pmatrix}X_{11}&\cdots&X_{16}\\
\vdots&\ddots&\vdots\\
X_{16}^{\eta*}&\cdots&X_{66}\end{pmatrix},
\end{align}
\begin{align}\label{hhequ043}
\widehat{Y}=T_{3}^{-1}YT_{3}^{-\eta*}=\begin{pmatrix}Y_{11}&\cdots&Y_{16}\\
\vdots&\ddots&\vdots\\
Y_{16}^{\eta*}&\cdots&Y_{66}\end{pmatrix}=\widehat{Y}^{\eta*},
\end{align}
be partitioned in accordance with (\ref{hhequ041}). Substituting $\widehat{X}$ and $\widehat{Y}$ of (\ref{hhequ042}) and (\ref{hhequ043}) into (\ref{hhequ041}) yields
\begin{align*}
\begin{pmatrix}\begin{smallmatrix}
X_{14}+X_{14}^{\eta*}+Y_{55}&X_{15}+X_{24}^{\eta*}&X_{34}^{\eta*}+Y_{45}^{\eta*}&X_{44}^{\eta*}+Y_{25}^{\eta*}&X_{54}^{\eta*}
&X_{11}+Y_{25}^{\eta*}&X_{12}+Y_{35}^{\eta*}&X_{13}&Y_{15}^{\eta*}&0&0\\
X_{24}+X_{15}^{\eta*}&X_{25}+X_{25}^{\eta*}&X_{35}^{\eta*}&X_{45}^{\eta*}&X_{55}^{\eta*}&X_{21}&X_{22}&X_{23}&0&0&0\\
X_{34}+Y_{45}&X_{35}&Y_{44}&Y_{24}^{\eta*}&0&X_{31}+Y_{24}^{\eta*}&X_{32}+Y_{34}^{\eta*}&X_{33}&Y_{14}^{\eta*}&0&0\\
X_{44}+Y_{25}&X_{45}&Y_{24}&Y_{22}&0&X_{41}+Y_{22}&X_{42}+Y_{23}&X_{43}&Y_{12}^{\eta*}&0&0\\
X_{54}&X_{55}&0&0&0&X_{51}&X_{52}&X_{53}&0&0&0\\
X_{11}^{\eta*}+Y_{25}&X_{21}^{\eta*}&X_{31}^{\eta*}+Y_{24}&X_{41}^{\eta*}+Y_{22}&X_{51}^{\eta*}&Y_{22}&Y_{23}&0&Y_{12}^{\eta*}&0&0\\
X_{12}^{\eta*}+Y_{35}&X_{22}^{\eta*}&X_{32}^{\eta*}+Y_{34}&X_{42}^{\eta*}+Y_{23}^{\eta*}&X_{52}^{\eta*}&Y_{23}^{\eta*}&Y_{33}&0&Y_{13}^{\eta*}&0&0\\
X_{13}^{\eta*}&X_{23}^{\eta*}&X_{33}^{\eta*}&X_{43}^{\eta*}&X_{53}^{\eta*}&0&0&0&0&0&0\\
Y_{15}&0&Y_{14}&Y_{12}&0&Y_{12}&Y_{13}&0&Y_{11}&0&0\\
0&0&0&0&0&0&0&0&0&0&0\\
0&0&0&0&0&0&0&0&0&0&0\end{smallmatrix}
\end{pmatrix}
\end{align*}
\begin{align}\label{hhequ045}
=\begin{pmatrix}
A_{11} & \cdots&   A_{19}&   A_{1,10}&0\\
\vdots& \ddots&   \vdots&\vdots&\vdots\\
A_{19}^{\eta*}& \cdots&A_{99}&A_{9,10}& 0\\
A_{1, 10}^{\eta*}& \cdots&A_{9,10}^{\eta*}&0&0\\0&\cdots&0 &0&\Sigma
\end{pmatrix}.
\end{align}

If the equation (\ref{4system002}) has a solution $(X,Y)$, then by (\ref{hhequ045}), we obtain that
\begin{align}\label{block01}
\Sigma=0,\quad \left(A_{1,10}^{\eta*},~  \cdots, ~A_{9,10}^{\eta*}\right)=0,\quad
A_{44}=A_{66},\quad A_{49}=A_{69},
\end{align}
\begin{align}\label{block02}
A_{29}=0,~A_{59}=0,~A_{89}=0,~A_{68}=0,~A_{78}=0,~A_{88}=0,~A_{35}=0,~A_{45}=0,~A_{55}=0.
\end{align}
and
\begin{align*}
X_{14}+X_{14}^{\eta*}+Y_{55}=A_{11},X_{15}+X_{24}^{\eta*}=A_{12},X_{34}^{\eta*}+Y_{45}^{\eta*}=A_{13},X_{44}^{\eta*}+Y_{25}^{\eta*}=A_{14},
X_{54}^{\eta*}=A_{15},
\end{align*}
\begin{align*}
X_{11}+Y_{25}^{\eta*}=A_{16},X_{12}+Y_{35}^{\eta*}=A_{17},X_{13}=A_{18},Y_{15}^{\eta*}=A_{19},X_{25}+X_{25}^{\eta*}=A_{22},X_{35}^{\eta*}=A_{23},
\end{align*}
\begin{align*}
X_{45}^{\eta*}=A_{24},X_{55}^{\eta*}=A_{25},X_{21}=A_{26},X_{22}=A_{27},X_{23}=A_{28},Y_{44}=A_{33},Y_{24}^{\eta*}=A_{34},
\end{align*}
\begin{align*}
X_{31}+Y_{24}^{\eta*}=A_{36},X_{32}+Y_{34}^{\eta*}=A_{37},X_{33}=A_{38},Y_{14}^{\eta*}=A_{39},Y_{22}=A_{44},X_{41}+Y_{22}=A_{46},
\end{align*}
\begin{align*}
X_{42}+Y_{23}=A_{47},X_{43}=A_{48},Y_{12}^{\eta*}=A_{49},X_{51}=A_{56},X_{52}=A_{57},X_{53}=A_{58},
\end{align*}
\begin{align*}
Y_{22}=A_{66},Y_{23}=A_{67},Y_{12}^{\eta*}=A_{69},Y_{33}=A_{77},Y_{13}^{\eta*}=A_{79},Y_{11}=A_{99}.
\end{align*}
Hence, the general solution $(X,Y)$ can be expressed as (\ref{hhequX}) and (\ref{hhequY}) by (\ref{hhequ045}).

Conversely, assume that the equalities in (\ref{block01}) and (\ref{block02}) hold, then by (\ref{hhequ042})-(\ref{hhequ045}), it can be
verified that the matrices have the forms of (\ref{hhequX}) and (\ref{hhequY}) is a solution of (\ref{hhequ045}), i.e., (\ref{4system002}).

We now want to prove that (\ref{rank03})-(\ref{rank04}) $\Longleftrightarrow$ (\ref{block01}) and (\ref{block02}). From $S_{A},S_{B},S_{C},$ and $S_{D}$ in Theorem \ref{theorem01}, we can infer that
\begin{align*}
r(A,B,C^{\eta*},D)=r(B,C^{\eta*},D) \Longleftrightarrow \left(A_{1,10}^{\eta*},~  \cdots, ~A_{9,10}^{\eta*}\right)=0,~\Sigma=0,
\end{align*}
\begin{align*}
r\begin{pmatrix}A&B&C^{\eta*}\\D^{\eta*}&0&0\end{pmatrix}=r(B,C^{\eta*})+r(D)\Longleftrightarrow A_{29}=0,~A_{89}=0,~A_{49}=A_{69},~\Sigma=0,
\end{align*}
\begin{align*}
r\begin{pmatrix}A&B&D\\B^{\eta*}&0&0\end{pmatrix}=r(B,D)+r(B)\Longleftrightarrow A_{68}=0,~A_{78}=0,~A_{88}=0,~A_{89}=0,~\Sigma=0,
\end{align*}
\begin{align*}
r\begin{pmatrix}A&C^{\eta*}&D\\C&0&0\end{pmatrix}=r(C^{\eta*},D)+r(C)\Longleftrightarrow A_{35}=0,~A_{45}=0,~A_{55}=0,~A_{59}=0,~\Sigma=0,
\end{align*}
\begin{align*}
r\begin{pmatrix}
A&0&B&0&D\\
0&-A&0&C^{\eta*}&D\\
B^{\eta*}&0&0&0&0\\
0&C&0&0&0\\
D^{\eta*}&D^{\eta*}&0&0&0\end{pmatrix}=2r\begin{pmatrix}B&0&D\\0&C^{\eta*}&D\end{pmatrix}
\Longleftrightarrow A_{44}=A_{66}=0,~\Sigma=0.
\end{align*}

\end{proof}

Next we give an example to illustrate Theorem \ref{theorem05}

\begin{example}
Given the real quaternion matrices:
\begin{align*}
B=\begin{pmatrix}1+\mathbf{j}&\mathbf{i}+\mathbf{k}&1+2\mathbf{i}+\mathbf{j}&-1-\mathbf{k}\\
\mathbf{i}-\mathbf{j}&-1-\mathbf{k}&-2+\mathbf{i}-\mathbf{j}&-\mathbf{i}+\mathbf{k}\end{pmatrix},
C=\begin{pmatrix}\mathbf{i}+\mathbf{j}&-2+\mathbf{k}\\
1+2\mathbf{j}&2\mathbf{i}+2\mathbf{k}\\
-\mathbf{i}+\mathbf{j}+\mathbf{k}&2-\mathbf{j}+\mathbf{k}\\
\mathbf{j}&\mathbf{k}\end{pmatrix},
\end{align*}
\begin{align*}
D=\begin{pmatrix}\mathbf{i}+\mathbf{j}&1+3\mathbf{i}&1+\mathbf{k}\\
-1+\mathbf{k}&-3+\mathbf{i}&\mathbf{i}-\mathbf{j}\end{pmatrix},
A=A^{\mathbf{i}*}=\begin{pmatrix}-16-6\mathbf{j}+34\mathbf{k}&9+17\mathbf{i}-31\mathbf{j}-3\mathbf{k}\\
9-17\mathbf{i}-31\mathbf{j}-3\mathbf{k}&-30+12\mathbf{j}-16\mathbf{k}\end{pmatrix}.
\end{align*}
Now we consider the $\mathbf{i}$-Hermitian solution to the real quaternion matrix equation (\ref{4system002}). Check that
\begin{align*}
r(A,B,C^{\eta*},D)=r(B,C^{\eta*},D)=2,
\end{align*}
\begin{align*}
r\begin{pmatrix}A&B&C^{\eta*}\\D^{\eta*}&0&0\end{pmatrix}=r(B,C^{\eta*})+r(D)=3,
\end{align*}
\begin{align*}
r\begin{pmatrix}A&B&D\\B^{\eta*}&0&0\end{pmatrix}=r(B,D)+r(B)=4,
\end{align*}
\begin{align*}
r\begin{pmatrix}A&C^{\eta*}&D\\C&0&0\end{pmatrix}=r(C^{\eta*},D)+r(C)=4,
\end{align*}
\begin{align*}
r\begin{pmatrix}
A&0&B&0&D\\
0&-A&0&C^{\eta*}&D\\
B^{\eta*}&0&0&0&0\\
0&C&0&0&0\\
D^{\eta*}&D^{\eta*}&0&0&0\end{pmatrix}=2r\begin{pmatrix}B&0&D\\0&C^{\eta*}&D\end{pmatrix}=8.
\end{align*}
All the rank equalities in (\ref{rank03})-(\ref{rank04}) hold. Hence, the real quaternion matrix equation (\ref{4system002}) has a  solution $(X,Y)$, where $Y$ is $\mathbf{i}$-Hermitian. Note that
\begin{align*}
X=\begin{pmatrix}
2+\mathbf{i}+\mathbf{k}&1+\mathbf{i}+\mathbf{j}&1&\mathbf{i}+\mathbf{k}\\
-1+\mathbf{k}&-\mathbf{i}+\mathbf{k}&\mathbf{j}&1\\
1+\mathbf{i}+\mathbf{j}+\mathbf{k}&1&1+\mathbf{j}&1+\mathbf{i}+\mathbf{k}\\
\mathbf{i}+\mathbf{j}+2\mathbf{k}&1-\mathbf{i}+\mathbf{k}&1+2\mathbf{j}&2+\mathbf{i}+\mathbf{k}
\end{pmatrix}
\end{align*}and
\begin{align*}
Y=Y^{\mathbf{i}*}=\begin{pmatrix}
1+\mathbf{j}&1+\mathbf{i}&\mathbf{j}\\
1-\mathbf{i}&\mathbf{k}&\mathbf{i}\\
\mathbf{j}&-\mathbf{i}&\mathbf{j}
\end{pmatrix}
\end{align*}
satisfy the real quaternion matrix equation (\ref{4system002}).
\end{example}

\section{\textbf{Conclusion}}

We have derived a simultaneous decomposition of four real quaternion matrices with the same row number
$(A,B,C,D),$ where $A=A^{\eta*}\in \mathbb{H}^{m\times m},
B\in \mathbb{H}^{m\times p_{1}},C\in \mathbb{H}^{m\times p_{2}},D\in \mathbb{H}^{m\times p_{3}}$. As applications of this simultaneous decomposition, we have presented necessary and sufficient conditions for the existence and the general $\eta$-Hermitian solution to the real quaternion matrix equation (\ref{4system001}). We have also given necessary and sufficient conditions for the existence and the general solution to the real quaternion matrix equation (\ref{4system002}). Some numerical examples are presented to illustrate the results.

\end{document}